\documentclass[a4paper, oneside, 11pt, reqno]{amsart}
\usepackage[utf8]{inputenc}
\usepackage{mathtools}
\usepackage{amsfonts}
\usepackage{amsmath}
\usepackage{amssymb}
\usepackage{amsthm}
\usepackage[bottom=3cm]{geometry}
\usepackage{mathrsfs} 
\usepackage[dvipsnames]{xcolor}
\usepackage{dsfont}
\usepackage{nicefrac}
\allowdisplaybreaks
\usepackage{tikz}
\usepackage{caption} 

\usepackage[pagebackref=true,colorlinks=true, linkcolor=blue, citecolor=blue]{hyperref}
\renewcommand*\backref[1]{\ifx#1\relax \else (Cited on #1) \fi}
\usepackage[nameinlink]{cleveref} 
\usepackage{orcidlink}

\usepackage{appendix}
\usepackage{enumerate}
\usepackage{tikz-cd} 

\theoremstyle{plain}
\newtheorem{theorem}{Theorem}
\newtheorem{proposition}[theorem]{Proposition}
\newtheorem{lemma}[theorem]{Lemma}

\newtheorem{definition}[theorem]{Definition}
\newtheorem{remark}[theorem]{Remark}

\newtheorem{example}[theorem]{Example}

\numberwithin{theorem}{section}
\numberwithin{equation}{section} 

\newcommand{\vertiii}[1]{{\left\vert\kern-0.25ex\left\vert\kern-0.25ex\left\vert #1 
    \right\vert\kern-0.25ex\right\vert\kern-0.25ex\right\vert}}

\newcommand{\Z}{\mathbb{Z}}

\newcommand{\R}{\mathbb{R}}
\newcommand{\N}{\mathbb{N}}
\newcommand{\X}{\mathbb{X}}
\newcommand{\Y}{\mathbb{Y}}

\newcommand{\calT}{\mathcal{T}}

\newcommand{\bbE}{\mathbb{E}}
\newcommand{\bbP}{\mathbb{P}}

\newcommand{\eps}{\varepsilon}
\newcommand{\ud}{u'}

\title{Convex order and faster transmission in first contact percolation}

\author{Benedikt Jahnel  \orcidlink{0000-0002-4212-0065} }
\address{Technische Universit\"at Braunschweig \& Weierstrass Institute Berlin}
\email{benedikt.jahnel@tu-braunschweig.de}
\author{Jonas K\"oppl  \orcidlink{0000-0001-9188-1883}}
\address{Technische Universit\"at Braunschweig}
\email{jonas.koeppl@tu-braunschweig.de}
\author{Lukas L\"uchtrath  \orcidlink{0000-0003-4969-806X}}
\address{Weierstrass Institute Berlin}
\email{lukas.luechtrath@wias-berlin.de}
\author{Anh Duc Vu \orcidlink{0009-0005-6913-4992}}
\address{Weierstrass Institute Berlin}
\email{anhduc.vu@wias-berlin.de}

\date{\today}

\begin{document}
\begin{abstract}
Inspired by strict-monotonicity criteria for the time constant in first passage percolation, we investigate convex ordering of point processes in relation to the time constant in first contact percolation. In a nutshell, first contact percolation models the spread of an infection as a contact process without recovery based on a generalized graphical representation, where the usual homogeneous Poisson point processes on the edges are replaced by general simple point processes. Based on a notion of convex ordering for point processes, we prove monotonicity in the number and existence of infection paths. We argue that this convex ordering is however not enough to ensure strict monotonicities in the asymptotic speed of the infection. Instead, we propose a criterion based on an ordering of void probabilities and prove a speed-up for one-dimensional systems based on $\Z$-stationary point processes.
\end{abstract}

\maketitle

\section{Introduction}
{\em First passage percolation} (FPP) is a classical pure-growth model where information is transmitted along the edges of a graph according to i.i.d.\ transmission times~\cite{auffinger201750,kesten2006aspects}. Under suitable conditions on the transmission-time distribution and the underlying graph, it is known that the system satisfies a shape theorem, that roughly says the following: Starting from a source node, the set of reachable vertices up to time $t\ge 0$, divided by $t$, converges almost surely to a non-trivial limiting set. In particular, information transport has an asymptotically linear speed, where the associated time constant depends on the spatial direction in which one tracks the information. 

\medskip
Considering two different transmission-time distributions, it is clear that if they are stochastically ordered, by a monotone coupling, also the associated time constants in FPP must be ordered. What is more challenging is to find general conditions on the transmission-time distributions such that they lead to ordered time constants. A first answer to this question is given in~\cite{van1993inequalities} in case of FPP on the hypercubic lattice with the help of the notion of convex ordering. Two probability measures $\mu,\nu$ on $\R_{\ge0}$ are convexly ordered if
\begin{align*}
    \nu(f)\le \mu(f),\qquad\text{ for all }f\text{ convex and decreasing}. 
\end{align*}
In this case, $\mu$ is {\em more variable} compared to $\nu$ and for example the first moments of the distributions are automatically ordered. In~\cite{van1993inequalities} it is proven that convex ordering already implies the ordering of the time constants and, under an additional usefulness criterion, even implies a strict ordering of the time constants, at least in dimensions $d\ge2$. Some additional requirements on the measures could be lifted in the planar case in~\cite{marchand2002strict}. 

\medskip
An alternative to FPP has been proposed in~\cite{jahnel2026first} based on the idea that messages are exchanged between neighboring vertices of a graph at random meeting times. More precisely, this model of {\em first contact percolation} (FCP) can be seen as a generalized graphical representation, where the i.i.d.\ Poisson point processes, as used in the Richardson model of FPP, are replaced by general random closed sets. Similar to the contact process without recovery, an infection now travels along increasing sequences of meeting times starting in a source node. If the random closed sets have an $\R$-stationary distribution, then associated time constants can be recovered by an appropriate mapping to FPP. This is however no longer possible if the distribution of the random closed sets is for example only $\Z$-stationary, which may be suitable to model time-periodic behaviour as often observed in nature. In this case,~\cite{jahnel2026first} provides some case studies for models featuring varying degrees of reliability and their associated time constants and shape theorems. 

\medskip
Inspired by the existing criteria for inequalities of time constants in FPP mentioned above, in this work, we study related, but also different, criteria for the comparison of time constants in FCP. 
The manuscript is organized as follows. We present our setting, main results and examples in Section~\ref{sec_set}. In Section~\ref{sec_proofs}, we exhibit all proofs.

\section{Setting, main results and examples}\label{sec_set}
Consider the hypercubic lattice $\Z^d$ with $d\ge 1$ and nearest-neighbor edges. We equip each edge $e=\{x,y\}$ independently with a random closed set $\X_e$ on $\R$ that represents meeting times of the pair of end vertices $x$ and $y$; see~\cite{molchanov2005theory} for details on random closed sets. We assume that the $\X_e$ follow a common distribution represented by the distribution of $\X$.
Assume that at time $t=0$, an infection is present at the origin $o\in \Z^d$. We are interested in the distribution of {\em infected} vertices at time $t\ge 0$ given by 
$$
I(t):=\{y\in \Z^d\colon \exists \text{ permitted path }\Gamma_t(o,y)\},
$$
where a space-time path
$$
\Gamma_t(x,y)=\big((e_0,t_o),\dots, (e_n,t_n)   \big)
$$
is called {\em permitted} if $e_i=\{x_{i},x_{i+1}\}$ for a nearest-neighbor self-avoiding path $x=x_0,\dots,x_{n+1}=y$  connecting $y$ to $x\in\Z^d$ and 
$0\le t_0\le\dots \le t_n< t$ such that $t_i\in \X_{\{x_{i},x_{i+1}\}}$ for all $0\le i\le n$.
In words, the infected set is the set of vertices that can be reached from the origin via an increasing sequence of meeting times of nearest-neighbor vertices, see Figure~\ref{fig_1} for an illustration based on a simple point process $\X$. This is {\em first contact percolation} (FCP) as introduced in~\cite{jahnel2026first}. 
\begin{center}
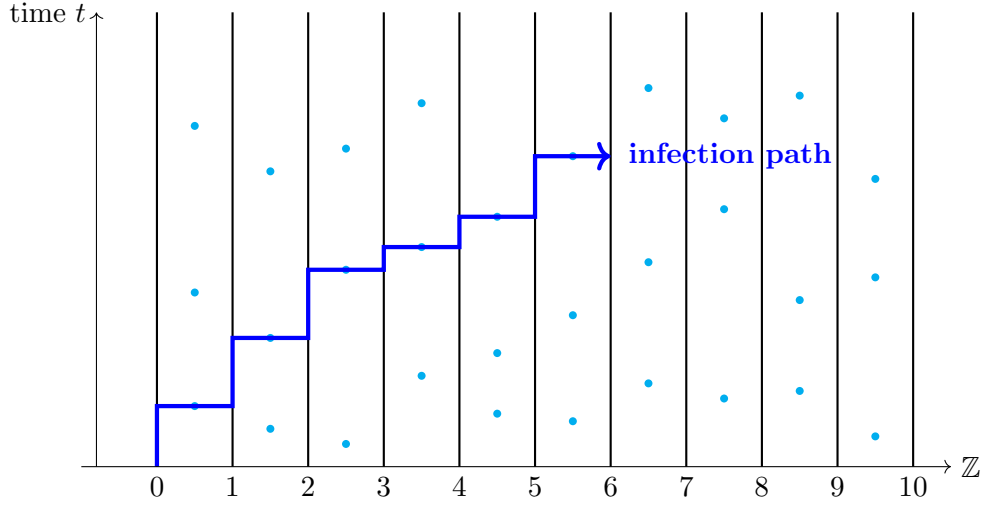
\begin{figure}[ht]
\begin{tikzpicture}[scale=1.0]
  \foreach \i in {0,...,10} {
    \draw[thick] (\i,0) -- (\i,6);
    \node[below] at (\i,0) {$\i$};
  }
  \draw[->] (-0.8,0) -- (-0.8,6) node[left] {time $t$};


  \foreach \t in {0.8,2.3,4.5} {
    \fill[cyan]     (0.5,\t) circle (1.5pt);
  }
  \foreach \t in {0.5,1.7,3.9} {
    \fill[cyan]      (1.5,\t) circle (1.5pt);
  }
  \foreach \t in {0.3,2.6,4.2} {
    \fill[cyan] (2.5,\t) circle (1.5pt);
  }
  \foreach \t in {1.2,2.9,4.8} {
    \fill[cyan]   (3.5,\t) circle (1.5pt);
  }
  \foreach \t in {0.7,1.5,3.3} {
    \fill[cyan]   (4.5,\t) circle (1.5pt);
  }
  \foreach \t in {0.6,2.0,4.1} {
    \fill[cyan]     (5.5,\t) circle (1.5pt);
  }
  \foreach \t in {1.1,2.7,5.0} {
    \fill[cyan]    (6.5,\t) circle (1.5pt);
  }
  \foreach \t in {0.9,3.4,4.6} {
    \fill[cyan]  (7.5,\t) circle (1.5pt);
  }
  \foreach \t in {1.0,2.2,4.9} {
    \fill[cyan]    (8.5,\t) circle (1.5pt);
  }
  \foreach \t in {0.4,2.5,3.8} {
    \fill[cyan]     (9.5,\t) circle (1.5pt);
  }

  \draw[->,ultra thick,blue]
    (0,0) -- (0,0.8)   
          -- (1,0.8)   
          -- (1,1.7)   
          -- (2,1.7)   
          -- (2,2.6)   
          -- (3,2.6)   
          -- (3,2.9)   
          -- (4,2.9)   
          -- (4,3.3)   
          -- (5,3.3)   
          -- (5,4.1)   
          -- (6,4.1);  

\node[blue,right] at (6,4.1) {\hspace{0.1cm}\bf infection path};
\draw[->] (-1,0) -- (10.5,0) node[right] {$\Z$};
\end{tikzpicture}
\caption{
  A realization of first contact percolation on the integer line $\mathbb{Z}$.
  Point processes mark potential meeting times on each edge, and the blue arrow traces one sample infection path through the network.
}
\label{fig_1}
\end{figure}
\end{center}

\subsection{Results with convex ordering}
We are interested in criteria that allow us to compare FCP based on different {\em simple point processes} $\X$ and $\X'$. Inspired by the convex-ordering criterion for FPP, as mentioned in the introduction, we present here an associated notion, see~\cite{blaszczyszyn2011clustering,blaszczyszyn2009directionally,meester1993regularity}. We write $\X(A)$ for the number of points of $\X$ in $A$.
\begin{definition}[Convex ordering of point processes]
    Given two simple point processes $\X,\X'$ on $\R$, we say that $\X$ is more {\em reliable} than $\X'$ if for every finite collection of Borel sets $B_1,\dots,B_k\in\mathcal{B}(\R)$ and every decreasing convex function $\varphi\colon\R^k\to\R$, we have
    \begin{align}\label{def:convexordering}
    \bbE[\varphi(\X(B_1),\dots,\X(B_k))] \leq \bbE[\varphi(\X'(B_1),\dots,\X'(B_k))].
    \end{align}
    In this case, we write $\X\gg\X'$.
\end{definition}
The notion {\em decreasing} is used with respect to the usual partial order on $\R^k$. 
We comment on examples of comparable models further below. 
\begin{remark}\label{rem:conv}
Let us make the following remarks.
\begin{enumerate}[(i)]
    \item The notion of convex ordering can be introduced even for random Radon measures on $\R^d$ replacing the point processes on $\R$. It is based on a multivariate version of the convex ordering as mentioned in the context of FPP, see for example~\cite{shaked2007stochastic}. 
    \item By~\cite[Proposition~3.1]{blaszczyszyn2009directionally}, it suffices to check~\eqref{def:convexordering} for mutually disjoint and bounded Borel sets.
    \item As noted in~\cite[Proposition~3.3]{blaszczyszyn2011clustering}, convex ordering implies that void probabilities are ordered in the sense that if $\X\gg\X'$, then for all bounded $B\in\mathcal{B}(\R)$,
\begin{align}\label{void}
    \bbP(\X(B)=0)\le \bbP(\X'(B)=0).
\end{align} 
\end{enumerate}

\end{remark}
With this criterion at hand, let us state our first result on the number of permitted paths, 
    \begin{align*}
N^\X_t (x,y):=\#\{\text{permitted paths }\Gamma_t(x,y)\}.
    \end{align*}
\begin{theorem}[Monotonicities in permitted paths]\label{thm:perm}
    Let $\X,\X'$ be simple point processes with $\X\gg\X'$. For all $x,y\in \Z^ d$ and $t\ge0$, we have
    $$\bbE[N^\X_t (x,y)]\geq \bbE[N^{\X'}_t (x,y)]\quad \text{ and }\quad \bbP(N^\X_t (x,y)\ge 1)
    \geq \bbP(N^{\X'}_t (x,y)\ge 1).$$
\end{theorem}
The proof and all other proofs are postponed to Section~\ref{sec_proofs}.

\medskip
Let us come back to the setting where $\X$ and $\X'$ are given as random closed sets.
For $v\in \R^d$, denote by
\[T_\X(v):=\inf\{t\ge 0\colon N_t^\X(o,[v])\ge 1\},\]
with $[v]=(\lfloor v_1\rfloor,\dots, \lfloor v_d\rfloor)$, the first hitting time of $v$.
If the almost-sure limits
\begin{align}\label{def:timeconstant}
\lim_{t\uparrow\infty}t^{-1}T_\X(ty)=c_\X(y)\quad\text{ and }\quad\lim_{t\uparrow\infty}t^{-1}T_{\X'}(ty)=c_{\X'}(y)
\end{align}
exist, then, under the conditions in and with the help of Theorem~\ref{thm:perm}, we have that the associated {\em time constants} \(c_\X(y),c_{\X'}(y)\) satisfy $c_\X(y)\le c_{\X'}(y)$. Indeed, for all $\varepsilon>0$, we have that
$$
\bbP\big(T_\X(ty)\le (c_{\X}(y)-\varepsilon)t\big)\ge \bbP\big(T_{\X'}(ty)\le (c_{\X}(y)-\varepsilon)t\big),
$$
where the left-hand side tends to zero. This implies $c_\X(y)\le c_{\X'}(y)$. 

\medskip
This is an indication that convex ordering is a strong condition in the setting of point processes. In fact at least under $\R$-stationarity, we see that convex ordering on the level of point processes in FCP implies {\em stochastic ordering} on the level of transition-time distributions in FPP. In particular, the transition-time distributions are {\em convexly ordered}. This is due to the following identification, as already highlighted in~\cite{jahnel2026first}. In case that $\X,\X'$ are $\R$-stationary, the time constants in FCP can be derived from the corresponding time constants in FPP with transition times distributed via
$$
\mu([0,s])=\bbP(\X([0,s])>0),\qquad s\ge 0,
$$
where the right-hand side is the size-biased hitting probability. But these hitting probabilities are equivalent to void probabilities as mentioned in Remark~\ref{rem:conv}. In particular,~\eqref{void} implies that the void probabilities of $\X,\X'$ can be coupled, and in turn also the transition times in FPP, and thus the ordering of the time constants is trivially achieved.

\medskip
On the other hand, if $\X$ is an unit-intensity Poisson point process and $\Y$ the inhomogeneous Poisson point process with intensity measure $(1+|1+x|^{-1})\mathrm{d}x$, then the associated FCP models on $\Z_{\ge 0}$ both have asymptotic speed $1$. However, $\X\ll\Y$ and in particular their void probabilities are even strictly ordered. This tells us that strict convex ordering alone is not suitable to give criteria for strict inequalities for time constants. 

\medskip
On top of these shortcomings, since convex ordering is a somewhat global property, many natural periodic point-process models are not comparable, as we discuss now in view of some examples already investigated in~\cite{jahnel2026first}. 
Recall that $\X$ is a simple point process on $\R$ and consider the following examples.
\begin{itemize}
    \item $\X={\rm L}$, if $\X$ is the fixed {\em lattice} $\Z$. 
    \item $\X={\rm SL}$, if $\X=\Z+U$ where $U$ is a single uniform random variable on $[0,1)$, in words,  $\X$ is the {\em stationarized lattice}. 
    \item $\X={\rm PL}$, if $\X=\bigcup_{z\in\Z}\{z+U_z\}$ where $(U_z)_{z\in \Z}$ is an i.i.d.\ family of uniform random variables on $[0,1)$, in words, $\X$ is a {\em perturbed lattice}.
    \item $\X={\rm SPL}$, if $\X=\bigcup_{z\in\Z}\{z+U_z\}+U'$ where $(U_z)_{z\in \Z},U'$ is an i.i.d.\ family of uniform random variables on $[0,1)$, in words, $\X$ is a {\em stationarized perturbed lattice}.
    \item $\X={\rm Poi}$, if $\X$ is a {\em Poisson point process} with intensity $1$.
\end{itemize}
With respect to convex ordering we can provide the following analysis.
\begin{lemma}\label{lem:ex}
The models mentioned above have the following relations with respect to convex ordering as in Definition~\ref{def:convexordering}.
    \begin{enumerate}
        \item {\rm L}, {\rm SL}, {\rm PL}, and {\rm SPL} are mutually non-comparable in terms of convex ordering.
        \item {\rm L}, {\rm SL}, {\rm Poi} are mutually non-comparable in terms of convex ordering.
        \item ${\rm PL\gg Poi}$ and ${\rm SPL\gg Poi}$.
    \end{enumerate}
\end{lemma}
In~\cite{jahnel2026first}, it is shown that time constants are strictly ordered for $d=1$ as 
$
c_{\rm L}<c_{\rm SL}<c_{\rm SPL}<c_{\rm Poi_1}
$
where, without loss of generality we set $c_{\cdot}=c_{\cdot}(1)$. 
On the other hand, in general dimensions, it is proved there that 
$
c_{\rm L}(y)\le c_{\rm SL}(y)\le c_{\rm PL}(y)\le c_{\rm Poi_1}(y)
$,
for all $y\in \R^d$,
despite the lack of convex ordering.
In this context, let us further mention that FCP models based on stationary Poisson point process can be mapped to Richardson-type FPP models. On the level of FPP, one can then use monotone couplings to give an alternative proof for the inequalities $c_{\rm SL}\le c_{\rm PL}\le c_{\rm Poi_1}$. 

Nevertheless, the references~\cite{blaszczyszyn2011clustering,blaszczyszyn2009directionally,meester1993regularity} provide a large class of models that are in fact comparable with respect to convex ordering, such as for example Cox point processes, determinantal, or permanental point processes etc..

\subsection{Strict speed-up via void criterion}
In order to derive a workable condition that ensures a strict speed-up, we limit our attention to FCP on $\Z_{\ge 0}$ with $\Z$-stationary $\X$ and $\X'$. In words, we consider a one-dimensional population with periodic meeting habits. 
We start by establishing existence and non-triviality of time constants, as defined in~\eqref{def:timeconstant}.

\begin{proposition}[Existence of time constants on $\Z_{\ge 0}$]\label{lem:Shape}
    Consider the FCP model on $\Z_{\geq0}$ based on a $\Z$-stationary random closed set $\X$. Then, there exists $c_\X\in [0,\infty]$ such that almost surely, 
$$
\lim_{t\uparrow\infty}t^{-1}T_\X(t)=c_\X.
$$
Further,    
    \begin{enumerate}
        \item if $\bbP(\X=\emptyset)>0$, then the FCP almost-surely spreads only to finitely many sites, in particular $c_\X=\infty$.
        \item FCP has a finite time constant $c_\X\leq M$ as long as
        \begin{equation}\label{eq:finite-waiting-time}
M:=\int_0^\infty\bbP(\X([0,s])=0)\mathrm{d} s
            < \infty.
        \end{equation}
        \item $c_\X=0$ if and only if $\bbP(x\in\X)=1$ for some $x\in\R$.
    \end{enumerate}
\end{proposition}
Note that in Item~$(3)$, we assume instantaneous transmissions in cases where consecutive edges have coincident meeting times. 

\medskip
Next, we state our second main result, which gives a criterion for a strict speed-up in one-dimensional FCP. 
We say that $\X$ {\em dominates} $\X'$ in terms of {\em hitting probabilities} if, for all intervals $[a,b]$,
\begin{align}\label{eq:voiddom}
\bbP(\X'([a,b])>0 )\le \bbP(\X([a,b])>0).
\end{align}
Also, recall that $\X$ has {\em no atoms} if and only if $\bbP(x\in\X)=0$ for every $x\in\R$.

\begin{theorem}[Strict speed-up on $\Z_{\ge 0}$]\label{thm:speedup}
    Let $\X,\X'$ be $\Z$-stationary random closed sets such that $\X$ dominates $\X'$ in terms of hitting probabilities. Additionally, assume that there exists an interval $[a,b]$,
    with $a\le b$, such that
\begin{equation}\label{eq:thm-condition}
        3\bbP(\X'([a,b])>0) < \bbP(\X([a,b])>0),
    \end{equation}
    as well as \eqref{eq:finite-waiting-time} for $\X$ and that $\X'$ has no atoms. Then, $c_\X<c_{\X'}$.
\end{theorem}
We believe that the constant $3$ in~\eqref{eq:thm-condition} is not optimal but we are not convinced that it should in fact be equal to $1$ without additional assumptions.
Let us note that, in order to prove speed-ups in FPP, there is the advantage that any gain persists. In FCP, the situation is different as temporal gains can easily get lost. This is highlighted in the following example, where even~\eqref{eq:thm-condition} is satisfied. In particular, the no-atoms assumption cannot be dropped entirely.
\begin{example}[No speed-up on $\Z_{\ge 0}$ with atoms]
    Let $\X'$ be a thinning of $\Z$ where each site $i\in\Z$ is kept with probability $p$ and deleted otherwise, and consider 
    $$\X:=\bigcup_{i\in\X'}\{i,i+1/2\}=\X'+\{0,1/2\}.$$
    Then, $\X,\X'$ are $\Z$-stationary with 
    $$\bbP(\X([1/3,2/3])>0)=p>0=\bbP(\X'([1/3,2/3])>0).$$
    However, we have
    $$T_{\X'}(t)=T_{\X}(t) \qquad\text{ for all } t\geq0,$$
    and, in particular, $c_\X=c_{\X'}$. Roughly speaking, the additional meetings in $\X$ are useless for the transmission of information but are still detected by the stochastic order.
\end{example}

\section{Proofs}\label{sec_proofs}
\subsection{Proofs related to convex ordering}
Let us start by lifting convex ordering to finite sequences of i.i.d.\ point processes. Here and in the following, we use concave increasing functions as test functions.
\begin{lemma}[Convex ordering for i.i.d.\ copies]\label{lem:coniid}
    Let $\X\gg\X'$ and consider i.i.d.\ copies $\X_1,\X_2,\dots$ of $\X$ and  i.i.d.\ copies $\X'_1,\X'_2,\dots$ of $\X'$. Then, for every finite collection of Borel sets $B_1,\dots,B_k\in\mathcal{B}(\R^d)$ and every increasing concave function $\varphi\colon\R^{k\times k}\to\R$, we have
        $$\bbE[\varphi((\X_i(B_j))_{i,j\leq k})] \geq \bbE[\varphi((\X'_i(B_j))_{i,j\leq k})] \,.$$
\end{lemma}
\begin{proof}
Note that if $\varphi\colon\R^{k\times k}\to\R$ is increasing and concave, then $\varphi_i\colon\R^{k}\to\R$, 
$x\mapsto \varphi(a_1,\dots, a_{i-1},x,a_{i+1},\dots, a_k)$ is increasing and concave for any $i\in \{1,\dots, k\}$ and $a_1,\dots, a_{i-1},a_{i+1},\dots, a_k\in \R^k$. But then the result follows by conditioning and independence.
\end{proof}
Next, we collect some basic properties of concave functions. \begin{lemma}[Properties of concave functions]\label{lem:concave-fcts-operations}
    Let $I$ be some index set and let $\varphi_i\colon\R^k\to\R\cup\{-\infty\}$ be concave for every $i\in I$. Then,
    \begin{enumerate}
        \item $\inf_{i\in I} \varphi_i$ is concave.
        \item If $\sharp I<\infty$, then $\sum_{i\in I} \varphi_i$ is concave.
    \end{enumerate}
    The same is true if we replace concave with increasing and concave.
\end{lemma}
\begin{proof}
For Item~(i), let $x,y\in\R^k$ and $t\in(0,1)$. Then, since all the $\varphi_i$ are concave
        $$\inf_{i\in I}\varphi_i\big(tx+(1-t)y\big) \geq \inf_{i\in I}\Big( t\varphi_i(x)+ (1-t)\varphi_i(y) \Big) \geq \inf_{i\in I} t \varphi_i(x) + \inf_{i\in I} (1-t) \varphi_i(y),$$
        which shows the claim.
Finally, the statement of Item (ii) follows immediately from applying the definition of concavity.
\end{proof}

We are now in the position to prove our first main result. 
\begin{proof}[Proof of Theorem~\ref{thm:perm}]
Let us fix $x,y\in \Z^d$, $t\ge 0$, write $N=N^{\X}_t(x,y)$, $N'=N^{\X'}_t(x,y)$ and recall that $\X(B)<\infty$ if $B$ is bounded. 

Consider any spatial path $\gamma=(e_1,\dots,e_k)$ in $\Z^d$ connecting $x$ to $y$. Then, the number $N(\gamma)$ of permitted paths under $\X$ up to time $t$ following along $\gamma$ can be obtained via the monotone limit
\begin{align}\label{eq:mindar}
N(\gamma) = \lim_{n\uparrow\infty} \sum_{(B_1,\dots,B_k)\in J_k(n)}\min_{j=1,\dots,k}\X_{e_j}(B_j),
\end{align}
    where $J_k(n)$ denotes the set of increasing sequences of $k$-many intervals of sidelength $t/n$, with $n\ge k$, i.e.,
    \[J_k(n):=\big\{ (B_1,\dots, B_k)\colon B_i=[(j_i-1)t/n,j_i t/n), \, j_i\in\{1,\dots,n\}\text{ with }j_1\le j_2\le \cdots\le j_n \big\}.\]
  Indeed, in order to see~\eqref{eq:mindar}, note that, for all $n\ge k$,
   \begin{equation*}
N(\gamma) = \sum_{(B_1,\dots,B_k)\in J_k(n)}\prod_{j=1,\dots,k}\X_{e_j}(B_j)\ge \sum_{(B_1,\dots,B_k)\in J_k(n)}\min_{j=1,\dots,k}\X_{e_j}(B_j),
\end{equation*}
which provides the lower bound. On the other hand, almost surely the minimal distance between the relevant meeting times is positive, i.e.,  $0<\kappa:=\min\{\kappa_j\colon j=1,\dots,k\}$, where $\kappa_j=\min\{|s_i-s_j|\colon s_i\neq s_j\in \X_{e_j}\cap [0,t)\}$ since $\X$ has no accumulation points. Hence, for any $n\ge k$ such that $1/n<\kappa$ we have that 
\begin{equation*}
\prod_{j=1,\dots,k}\X_{e_j}(B_j)=\min_{j=1,\dots,k}\X_{e_j}(B_j),
\end{equation*}
since $\X_{e_j}(B_j)\in \{0,1\}$ for all $j$. 

The expression on the right-hand side of~\eqref{eq:mindar} is useful since the function $\N_0^k\to \N_0$, $(n_1,\dots, n_k)\mapsto \min_{j=1,\dots,k}n_j$ is a minimum over increasing linear functions and thus increasing and concave by Lemma~\ref{lem:concave-fcts-operations} Item~(i). Let $K_L$ denote the set of spatial paths $\gamma\subset [-L,L]^d$ from $x$ to $y$. Then,  
        \begin{align*}
            \bbE[N] &\ge \sum_{\gamma\in K_L} \bbE[N(\gamma)] \\
            &=\bbE\Big[\lim_{n\uparrow\infty}\sum_{\gamma\in K_L} \sum_{(B_1,\dots,B_{|\gamma|})\in J_{|\gamma|}(n)}\min_{j=1,\dots,{|\gamma|}}\X_{e_j}(B_j)\Big] \\
&= \lim_{n\uparrow\infty} \sum_{\gamma\in K_L} \sum_{(B_1,\dots,B_{|\gamma|})\in J_{|\gamma|}(n)}\bbE\Big[\min_{j=1,\dots,{|\gamma|}}\X_{e_j}(B_j)\Big]\\
           &\ge  \lim_{n\uparrow\infty} \sum_{\gamma\in K_L} \sum_{(B_1,\dots,B_{|\gamma|})\in J_{|\gamma|}(n)}\bbE\Big[\min_{j=1,\dots,{|\gamma|}}\X'_{e_j}(B_j)\Big]\\
            &= \sum_{\gamma\in K_L} \bbE[N'(\gamma)],
        \end{align*}
    where we used dominated convergence in the third and fifth line and convex ordering as in Lemma~\ref{lem:coniid} in the fourth line. Here, the dominated convergence above is justified since, for any fixed $\gamma$ we can bound
\begin{align}\label{eq:dom}
\sum_{(B_1,\dots,B_k)\in J_k(n)}\min_{j=1,\dots,k}\X_{e_j}(B_j) \leq \sum_{j=1}^k\X_{e_j}([0,t]),
\end{align}
    which has expectation bounded from above by  $|\gamma|\,\bbE[\X([0,t])]$, which is finite by assumption.
Letting $L$ tend to infinity finishes this first part of the proof.

\medskip
The proof of the second claim is similar to the one above, as we only replace the sum over paths by a minimum as follows. Since $\min\{N,1\}$ is a Bernoulli experiment, for all $L>0$, we can bound
    \begin{align*}
            \bbP(N\ge 1)& =\bbE[\min\{N,1\}]\\ 
            &\ge \bbE\Big[\min\Big\{\sum_{\gamma\in K_L} N(\gamma),1\Big\}\Big] \\
            &=\bbE\Big[\lim_{n\uparrow\infty}\min\Big\{\sum_{\gamma\in K_L} \sum_{(B_1,\dots,B_{|\gamma|})\in J_{|\gamma|}(n)}\min_{j=1,\dots,{|\gamma|}}\X_{e_j}(B_j),1\Big\}\Big] \\
            &=\lim_{n\uparrow\infty}\bbE\Big[\min\Big\{\sum_{\gamma\in K_L} \sum_{(B_1,\dots,B_{|\gamma|})\in J_{|\gamma|}(n)}\min_{j=1,\dots,{|\gamma|}}\X_{e_j}(B_j),1\Big\}\Big] \\
            &\geq \lim_{n\uparrow\infty}\bbE\Big[\min\Big\{\sum_{\gamma\in K_L} \sum_{(B_1,\dots,B_{|\gamma|})\in J_{|\gamma|}(n)}\min_{j=1,\dots,{|\gamma|}}\X'_{e_j}(B_j),1\Big\}\Big] \\
&= \bbE\Big[\min\Big\{\sum_{\gamma\in K_L} N'(\gamma),1\Big\}\Big],
        \end{align*}
    where we used dominated convergence with bound~\eqref{eq:dom} in the fourth and sixth line and Lemma~\ref{lem:concave-fcts-operations} Items~(i) and~(ii) in the fifth line. Since $L$ was arbitrary, the claim follows. 
\end{proof}

\begin{proof}[Proof of Lemma~\ref{lem:ex}] We verify the individual relations. For Items~(i) and~(ii), we observe contradicting inequalities when we use $\varphi(x)=\sqrt{x}$ and apply it to the sets $\{0\},(0,1),(-1/2,1/2)\backslash\{0\},(-1/2,0)\cup(1/2,1)$. Item~(iii) follows from~\cite[Proposition 6.1]{blaszczyszyn2011clustering} since ${\rm Pois}_1$ can be written as
    $$U+\bigcup_{z\in\Z}\bigcup_{i=1}^{N_z} \{z+U_{i,z}\},$$
    where $U,U_{z,i}$ are independent uniform random variables on $(0,1)$ and the $N_z$ are independent Poisson random variables. The claim follows from $N_z$ being more variable than the constant function $\equiv1$. 
\end{proof}

\subsection{Proofs related to strict speed-ups}
\begin{proof}[Proof of Proposition~\ref{lem:Shape}]
We refer to the proof strategy in~\cite{jahnel2026first} where a shape theorem is proved of the particular $\Z$-stationary point processes given by $\X={\rm PL}$ (with $n$-many contacts per unit interval) on $\Z^d$ and based on subadditivity arguments. More precisely, inspecting the corresponding proof of~\cite[Lemma 3.2]{jahnel2026first}, we observe that four conditions have to be checked from which only Conditions $(iii)$ and $(iv)$ require a closer look. The key observation is that the space-time independence given and used in the case $\X={\rm PL}$ can be replaced by only asking for independence over different edges since we work on $\Z_{\ge 0}$. Indeed, on $\Z_{\ge 0}$, the shortest permitted path from $ty$ to $(t+1)y$ does not use edges between $o$ and $ty$ and hence Condition $(iii)$ is satisfied. The same reasoning applies in the law-of-large-numbers-type argument that establishes Condition~$(iv)$.

Next for Item~$(1)$, if $\bbP(\X=\emptyset)=p>0$, then eventually an edge as no meeting times and the process stops. In order to see this, note that the subadditive ergodic theorem employed in establishing existence of the time constant also provides the characterization
$$
c_\X=\inf_{t>1}t^{-1}\bbE[T(t)],
$$
where, in this case, $\bbE[T(t)]\ge \ell\bbP(T(1)\ge \ell )\ge \ell p$ for all $\ell>0$ and $t\ge 1$ and hence $\bbE[T(t)]=\infty$.

For Item $(2)$, using the same characterization, 
$$
c_\X\le \bbE[T(1)]=\int_0^\infty\bbP(\X([0,s])=0)\mathrm{d} s=M,
$$
as required. 

For Item~$(3)$, it is clear that by the $\Z$-stationarity there exists arbitrarily large time $x$ with $\bbP(x\in \X)=1$ and at this time, the infection travels without delay to infinity and thus $c_\X=0$. On the other hand, it suffices to prove $c_\X>0$ if $p(x):=\bbP(x\in \X)<1$ for all $x\in [0,1]$. Note that we can even assume that $\sup\{p(x)\colon x\in [0,1]\}=:u<1$. Indeed, there exists a sequence $(x_n)_{n\ge 1}\subset[0,1]$ such that $p(x_n)$ converges to $u$. Assume that $u=1$, then by the compactness of $[0,1]$, there exists a convergent subsequence, which we also denote by $(x_n)_{n\ge 1}$, with $x_n\to x\in [0,1]$ and $p(x_n)\to p(x)=1$. 
But this contradicts $\bbP(x\in \X)<1$. 
From this we can further deduce that for $\rho:=(1-u)/2$, there exists $\delta>0$ such that $\sup\{\bbP(\X([x,x+\delta))>0)\colon x\in [0,1]\}< u+\rho<1$. Indeed, by contradiction, assume that for all $n\in\N$, there exists $x_n\in[0,1]$ such that $u+\rho\le \bbP(\X([x_n,x_n+1/n))>0)\le \bbP(\X((x_n-1/n,x_n+1/n))>0)$. Then, $(x_n)_{n\ge 1}$ has a convergent subsequence that we also call $(x_n)_{n\ge 1}$, which has a  limit point $x\in[0,1]$. But then, $\bbP(\X((x_n-1/n,x_n+1/n))>0)\to \bbP(x\in \X)\le u$, as $n$ tends to infinity, a contradiction to $\bbP(\X([x_n,x_n+1/n))>0)\ge u+\rho$.
Now, for our main claim, it suffices to prove that there exists $\varepsilon>0$ such that $\limsup_{t\uparrow\infty}\bbP(T(t)<\varepsilon t)=0$, which is equivalent to 
$$
\limsup_{t\uparrow\infty}\bbP\Big(\sum_{i=1,\dots, \lfloor t\rfloor}\tau_i<\varepsilon t\Big)=0,
$$   
where the $\tau_i$ (recursively) denote the first hitting time on the $i$-th edge after $\tau_{i-1}$. Recall that the $\tau_i$ are not independent in general. By the exponential Markov inequality, for $t\in \N$,
\begin{align*}
\bbP\Big(\sum_{i=1,\dots, t}\tau_i<\varepsilon t\Big)\le {\rm e}^{\varepsilon t}\bbE\Big[\exp\Big(-\sum_{i=1,\dots, t}\tau_i\Big)\Big]={\rm e}^{\varepsilon t}\bbE\Big[\exp\Big(-\sum_{i=1,\dots, t-1}\tau_i\Big)\bbE[{\rm e}^{-\tau_{t}}|\tau_{t-1}]\Big],  
\end{align*}
where, for all $\delta>0$,
\begin{align*}
\bbE[{\rm e}^{-\tau_{t}}|\tau_{t-1}]=\int_{0}^\infty {\rm e}^{-s}\, \bbP(\X[\tau_{t-1},\tau_{t-1}+s)>0){\rm d} s\le\delta \bbP(\X[\tau_{t-1},\tau_{t-1}+\delta)>0)+{\rm e}^{-\delta}.
\end{align*}
Hence, for $\delta$ sufficiently small, 
\begin{align*}
\bbP\Big(\sum_{i=1,\dots, t}\tau_i<\varepsilon t\Big)\le \exp\Big(\big(\varepsilon +\log(\delta(u+\rho)+{\rm e}^{-\delta})\big)t\Big),  
\end{align*}
where, for potentially even smaller $\delta$, $u':=\delta(u+\rho)+{\rm e}^{-\delta}<1$. Thus, for $\varepsilon<-\log u'$, we arrive at the desired convergence to zero.
\end{proof}
The above arguments stay intact also in the case where the process is started not at time $0$ but at some other time $t_o\in \R$.

\medskip
Before we prove our second main result, let us introduce some notation. Let $a\in\R$, $t\geq0$, and $(u_i)_{i\leq k}$ be a finite collection of $u_i\in[0,1]$ and consider
    $$F_a(t):=\bbP(\X([a,a+t]>0) \qquad\text{and}\qquad G_a(u):=F_a^{-1}(u)=\inf\{t\geq0\colon u\leq F_a(t)\}.$$
    In words,  $F_a(t)$ is the {\em hitting probability} of $[a,a+t]$ under $\X$ and $G_a$ is its generalized {\em inverse} (i.e., the quantile function).

\medskip
Next, we establish a natural coupling of FCP on $\Z_{\geq0}$ by coupling $\X$ and $\X'$ on the following probability space. Consider the recursively defined times at which the hitting probabilities are above a prescribed sequence of probabilities $(u_i)_{i\leq k}$, i.e.,
    $$\tau_t(u_1,\dots,u_k):= G_{\tau_t(u_1,\dots,u_{k-1})}(u_k)+\tau_t(u_1,\dots,u_{k-1}) \qquad\text{where }\tau_t=t.$$
Now, let $(U_i)_{i\ge 0}$ be a family of i.i.d.\ uniform random variables on $[0,1]$. Then, the time $\calT_t(n)$ when $\X$ first reaches the vertex $n\in\Z$ started at time $t$ is given by
    $$\calT_t(n):=\tau_t(U_1,\dots,U_n)=G_{\calT_t(n-1)}(U_n) + \calT_t(n-1) \quad\text{ and  }\quad \calT_t(0)=t.$$
Note that in distribution, $T(t)=\calT_{0}(\lfloor t\rfloor)$. 
For $\X'$ we denote the same quantities as $F',G',\tau'$ and $\calT'_t$. Note however that both $\calT_t$ and $\calT'_t$ use the same $(U_i)_{i\ge 0}$, this is our coupling.
    
    \medskip
The following lemma collects key properties of the coupling that will allow us later to establish the strict speed-up.

\begin{lemma}[Properties of the coupling]\label{lem:inverse}
    Let $a\in\R$, $t\geq0$, and $u,u_1,\dots,u_k\in[0,1]$. Couple $\X$ and $\X'$ as described above and assume that $\X$ dominates $\X'$ in terms of hitting probabilities as in~\eqref{eq:voiddom}. We have the following properties.
    \begin{enumerate}
        \item  $t\mapsto F_a(t) $ and $u\mapsto G_a(u)$ are increasing and right-continuous.
        \item $F'_a(t)\leq F_a(t)$ and $G_a(u)\leq G'_a(u)$.
        Further $\tau_{t}(u_1,\dots,u_k)\leq\tau'_{t}(u_1,\dots,u_k)$ and $\tau$ is increasing in all arguments $t,u_1,\dots, u_k$. In particular, $\calT_{t}(n)\leq\calT'_{t}(n)$ almost surely.
        \item Removing steps makes things faster, i.e., for each $i\in\{1,\dots,k\}$ we have that
        $$\tau_t(u_1,\dots,u_k)\geq \tau_t(u_1,\dots,u_{i-1},u_{i+1},\dots,u_k).$$
        \item Assume that $\X$ is $\Z$-stationary and that $\sup_{a\in[0,1]}\bbP(a\in\X)=u'<1$, i.e., $\X$ has no atoms of mass greater than $u'$. Then, for every $u>u'$ we have that
        $$\lim_{k\uparrow\infty} \tau_t(\underbrace{u,\dots,u}_k)=\infty.$$
        \item Let $u_1+u_2\leq1$. Then, $\tau_t(u_1,u_2)\leq\tau_t(u_1+u_2)$.
        \item Assume $\bbP(a\in \X)=0$ for all $a\in \R$. Then, $t\mapsto F_a(t)$ and $a\mapsto F_a(t)$ are continuous for all $a\in\R$ and $t\ge 0$, respectively.
    \end{enumerate}
\end{lemma}
We give the proof at the end of the section.

\begin{proof}[Proof of Theorem~\ref{thm:speedup}]
We can restrict us to the case where also 
$\X'$ satisfies~\eqref{eq:finite-waiting-time}, since otherwise there is nothing to show. First, we set
$$\ud:= \bbP(\X'([a,b])>0),$$
with $u'\in [0,1)$.
Then, by~\eqref{eq:thm-condition}, we have that $\bbP(\X([a,b])>0)=3(u'+\eps)$, where
\begin{align}\label{eq:3ud+3ueps}
3\eps=\bbP(\X([a,b])>0)-3\bbP(\X'([a,b])>0)=\bbP(\X([a,b])>0) - 3\ud>0.
\end{align}
By continuity of $F'_a(t)$, see Lemma~\ref{lem:inverse} Item~$(6)$, since $\X'$ is assumed to be atomless, and the fact that $\X'$ puts no mass on the emptyset, we find $a_o<a$ such that 
    \begin{align*}
    \bbP(\X'([a_o,b])>0)=2\ud+\eps.
    \end{align*}
    Note that by subadditivity, this implies
    \begin{align*}
        \bbP(\X'([a_o,a])>0)\geq \ud+\eps.
    \end{align*}
   Next, consider an adjusted minimal number of steps such that the hitting time is not too small,
\begin{align*}
\underline{m}&:=\underline{m}(\ud):=\min \{k\in\N\colon \tau'_0(\underbrace{\ud,\dots,\ud}_k)\geq 2\}\qquad\text{ and }\qquad
m:=2\underline{m}+1.
\end{align*}
Note that $m$ is almost-surely finite since $\X'$ has no atoms and we can use Lemma~\ref{lem:inverse} Item~$(4)$.
    Given $z\in\Z_{\geq0}$, we can now define sufficiently long intervals $I_z:=(m^2z, m^2(z+1)]\cap\Z$ and associated events
    \begin{align*}
        A_z:=\{U_i\in(\ud,\ud+\eps)\,\forall i\in I_z \text{ and } (U_i)_{i\in I_z} \text{ decreasing}\},
    \end{align*}
    which encode that the coupling variables $U_i$ are in a good interval along the entire edge interval $I_z$ and ordered. 
    Note that the probability of the event can be computed explicitly and is given by $\bbP(A_z)=\eps^{m^2}/(m^2)!>0$. 

\medskip
{\bf Claim 1:} {\em 
Let $t_o\in[a_o,a]$, consider $\ud+\eps\geq u_1\geq \dots \geq u_k \geq \ud$ for some $k\geq3$ and assume $\tau'_{t_o}(u_1)>a$. Then, 
    \begin{align}\label{eq:edgegain}
\tau_{t_o}(u_1,\dots,u_k)
            \leq \tau'_{t_o}(u_1, u_2,\dots,u_{k-1}).
    \end{align}}

In words, this shows that, under the choice of uniforms, the model based on $\X$ has arrived at least one vertex farther compared to the model based on $\X'$, without using more time.
In order to see~\eqref{eq:edgegain}, using the monotonicities exhibited in Lemma~\ref{lem:inverse} Item~$(2)$, we have that
    \begin{align*}
            \tau_{t_o}(u_1,\dots,u_k) \leq \tau_a(u_1,\dots,u_k)
            \leq \tau_a(\ud+\eps,\ud+\eps,\ud+\eps,u_4,\dots,u_k).
    \end{align*}
Using Lemma~\ref{lem:inverse} Item~$(5)$ and \eqref{eq:3ud+3ueps}, i.e., $\bbP(\X([a,b])>0)=3(u'+\varepsilon)$, we get
        \begin{align*}
\tau_a(\ud+\eps,\ud+\eps,\ud+\eps,u_4,\dots,u_k)
\leq\tau_a(3(\ud+\eps),u_4,\dots,u_k)
            = \tau_b(u_4,\dots,u_k).
    \end{align*} 
Next, again by Lemma~\ref{lem:inverse} Item~$(2)$ as well as since $\tau'_a(u')=b$ and $\tau'_{t_o}(u_1)>a$ by assumption, we have 
        \begin{align*}
\tau_b(u_4,\dots,u_k)
            &\le \tau'_b(u_4,\dots,u_k)= \tau'_a(\ud,u_4,\dots,u_{k})\leq \tau'_{t_o}(u_1,\ud,u_4,\dots,u_{k}).
    \end{align*}
    Finally, by the monotonicity, we arrive at
        \begin{align*}
            \tau'_{t_o}(u_1,\ud,u_4,\dots,u_{k})
            &\leq \tau'_{t_o}(u_1, u_2,\dots,u_{k-1}),
    \end{align*}
as desired.

\medskip
With this observation at hand, we now prove that on the events $A_z$, the process based on $\X$ gains a time advantage of at least one time unit compared to the $\X'$-based process.    

\medskip
{\bf Claim 2:} {\em 
Let $t_o\geq0$ and assume that $(U_{m^2z+1},\dots, U_{m^2(z+1)})\in A_z$ and that 
    $$\tau_{t_o}(u_1,\dots,u_{m^2z})+\ell \le \tau'_{t_o}(u_1,\dots,u_{m^2z})=:s_o,$$
for some $\ell\in \N_0$. Then,  
$$
    \tau_{t_o}(u_1,\dots,u_{m^2z}, U_{m^2z+1},\dots, U_{m^2(z+1)})+\ell+1
    \leq \tau'_{t_o}(u_1,\dots,u_{m^2z}, U_{m^2z+1},\dots, U_{m^2(z+1)}).
$$
}

Again in words, when we condition on the event that, up to the edge interval $I_z$, the $\X$-based process needs $\ell\ge 0$ time units less compared to the $\X'$-based process, then during the edge interval $I_z$ an additional time unit is saved. 
To see this, by $\Z$-stationarity, we only need to consider the case $\ell=0$. Furthermore,
    \begin{align*}
    \tau_{t_o}(u_1,\dots,u_{m^2z},U_{m^2z+1},\dots,U_{m^2(z+1)}) &= \tau_{\tau_{t_o}(u_1,\dots,u_{m^2z})}(U_{m^2z+1},\dots,U_{m^2(z+1)})\\
    &\leq \tau_{s_o}(U_{m^2z+1},\dots,U_{m^2(z+1)}),
    \end{align*}
    so it suffices to show
    \begin{align}\label{eq:4567}
    \tau_{s_o}(U_{m^2z+1},\dots,U_{m^2(z+1)}) \leq \tau'_{s_o}(U_{m^2z+1},\dots,U_{m^2(z+1)}) - 1.
    \end{align}
    Without loss of generality, we only consider the case $z=0$ and $s_o=t_o$ since this reliefs us from cumbersome index notations. For $1\le k\le m-1$, we consider special indices
    $$i_k:=\min\{i\in((k-1)m,km-1]\colon \exists z_k\in\Z \text{ s.t.}\, \calT'_{t_o}(i)\in[a_o,a]+z_k\text{ and }\calT'_{t_o}(i+1)>a+z_k\},$$
    in particular $i_{k+1}-i_k\geq2$.
    In words, $i_k$ is the first edge in the $k$-th subinterval of edges in $I_0$ such that the associated arrival time is in a $\Z$-shift of $[a_o,a]$ and the edge right after this is passed in the relevant area where the $\X$-based process is strictly faster. Note that already in each slightly shortened  subinterval $((k-1)m,(k-1)m+\underline{m}]$, the process spends at least one time unit since 
        \begin{equation}\label{eq:8989}
        \begin{split}
    \calT'_{t_o}((k-1)m+\underline{m})-\calT'_{t_o}((k-1)m)&\ge \tau'_{\calT'_{t_o}((k-1)m)}(\underbrace{\ud,\dots,\ud}_{\underline{m}})-\calT'_{t_o}((k-1)m)\\
   &\ge \tau'_{\lfloor\calT'_{t_o}((k-1)m)\rfloor}(\underbrace{\ud,\dots,\ud}_{\underline{m}})-\calT'_{t_o}((k-1)m)\\
      &\ge \tau'_{0}(\underbrace{\ud,\dots,\ud}_{\underline{m}})-1\ge 1,
      \end{split}
    \end{equation}
where we used the $\Z$-stationarity and the definition of $\underline{m}$. Note that the bound holds even for $k=m$ and this and the fact that we have a lower bound by $1$ are used later again. 
Inequality~\eqref{eq:8989} implies that, in order to avoid $[a_o,a]+\Z$, there must exist $i\in((k-1)m,(k-1)m+\underline{m}]$ and $\ell\in \N_0$ such that $\calT'_{t_o}(i)<a_o+\ell$ and $\calT'_{t_o}(i+1)>a +\ell$, which is equivalent to
$$
G'_{\calT'_{t_o}(i)}(U_{i+1})>a+\ell-\calT'_{t_o}(i).
$$
But this is equivalent to $\bbP(\X'([\calT'_{t_o}(i),a+\ell])>0)<U_{i+1}$, 
 where the right-hand side is bounded from above by $\ud+\eps$ since we work on the event $A_0$. The left-hand side however is bounded from below by  
$\bbP(\X'([a_o+\ell,a+\ell])>0)=\bbP(\X'([a_o,a])>0)\ge \ud+\eps$, where we used $\Z$-stationarity, which is a contradiction so that there exists an $i\in((k-1)m,(k-1)m+\underline{m}]$ such that $\calT'_{t_o}(i)\in[a_o,a]+\ell$.
We choose $\ell$ to be minimal, i.e., such that $a+\ell\geq\calT'_{t_o}(i)>a+\ell-1$. Then we see, with the same arguments as in~\eqref{eq:8989}, that $\calT'_{t_o}(i+\underline{m})\geq \calT'_{t_o}(i)+1 > a+\ell$. Thus, since $m=\max\{2\underline{m},4\}$ there exist $i_k\in\{i,\dots,i+\underline{m}\}\subset((k-1)m,(k-1)m+2\underline{m}]=((k-1)m,km-1]$ satisfying the properties we claimed on $i_k$ and  $i_k$ is finite for every $1\le k\le m-1$. 

In order to verify Claim~2, next, we inductively show that
\begin{align}\label{eq:1616}
\tau_{t_o}(U_{1},\dots,U_{i_k+k-1}) \leq \tau'_{t_o}(U_{1},\dots,U_{i_k}).
\end{align}
    The case of $k=1$ immediately follows from the monotone coupling. For $k+1$, we will use \eqref{eq:edgegain} to get
    \begin{equation}\label{eq:edgegain-applied}
        \tau_{\tau'_{t_o}(U_{1},\dots,U_{i_k})}(U_{i_k+1},\dots,U_{i_{k+1}+1}) \leq \tau'_{\tau'_{t_o}(U_{1},\dots,U_{i_k})}(U_{i_k+1}\dots,U_{i_{k+1}}).
    \end{equation}
    The conditions of \eqref{eq:edgegain} apply as follows: By the definition of  $i_k$, we have $\tau'_{t_o}(U_{1},\dots,U_{i_k})\in[a_0,a]+z_k$ for some $z_k\in\Z$, $\tau'_{t_o}(U_{1},\dots,U_{i_{k+1}})>a+z_k$, and monotonicity of $U_{i_k+1}\geq\dots\geq U_{i_{k+1}+1}$ due to $A_0$. These are indeed at least three terms as needed since $i_{k+1}-i_k\geq2$. Thus, 
    we can bound 
    \begin{align*}
        \tau_{t_o}(U_1,\dots,U_{i_{k+1}+k}) &= \tau_{\tau_{t_o}(U_1,\dots,U_{i_k+k-1})}(U_{i_k+k},\dots,U_{i_{k+1}+k})\\
        &\leq \tau_{\tau'_{t_o}(U_1,\dots,U_{i_k})}(U_{i_k+k},\dots,U_{i_{k+1}+k})\\
        &\leq \tau_{\tau'_{t_o}(U_{1},\dots,U_{i_k})}(U_{i_k+1},\dots,U_{i_{k+1}+1})\\
        &\leq \tau'_{\tau'_{t_o}(U_{1},\dots,U_{i_k})}(U_{i_k+1},\dots,U_{i_{k+1}})\\
        &=\tau'_{t_o}(U_1,\dots,U_{i_{k+1}}),
    \end{align*}
    where we used the induction hypothesis for the second line, the monotonicity due to $A_0$ in the third line and~\eqref{eq:edgegain-applied} in the forth line.
    Thus, using monotonicity of $U_i$ on the event $A_0$ and the definition of $m$, we can show
    \begin{align}\label{eq:6767}
    \tau_{t_o}(U_1,\dots,U_{m^2}) \leq \tau'_{t_o}(U_{1},\dots,U_{m^2-m+1}) \leq \tau'_{t_o}(U_{1},\dots,U_{m^2}) -1.
    \end{align}
    which implies~\eqref{eq:4567}.
   For the first inequality in~\eqref{eq:6767}, using~\eqref{eq:1616} and the monotonicity in the uniform random variables guaranteed by $A_0$, we have  
\begin{align*}
\tau_{t_o}(U_1,\dots,U_{m^2})&= \tau_{\tau_{t_o}(U_1,\dots,U_{i_{m-1}+m-2})}(U_{i_{m-1}+m-1},\dots,U_{m^2})\\
&\le \tau_{\tau'_{t_o}(U_1,\dots,U_{i_{m-1}})}(U_{i_{m-1}+m-1},\dots,U_{m^2})\\
&\le \tau_{\tau'_{t_o}(U_1,\dots,U_{i_{m-1}})}(U_{i_{m-1}+1},\dots,U_{m^2-m+1})\\
&\le \tau'_{t_o}(U_1,\dots,U_{m^2-m+1}).
\end{align*}
The second inequality in~\eqref{eq:6767} follows by the same arguments as~\eqref{eq:8989}, noting that $m-1\geq \underline{m}$.

\medskip
{\bf Claim 3:} {\em 
Finally, for the speed-up it suffices to exhibit $\delta>0$ such that under the coupling,
$$
\limsup_{n\uparrow\infty}\bbP(\calT'(n)-\calT(n)\le \delta n)=0,
$$
}
For this, by the stochastic domination of the hitting times, almost surely, $\calT'(n)-\calT(n)\ge 0$. Further, due to the almost-sure existence of the time constants, we may change the time scale and prove instead, 
$$
\limsup_{z\uparrow\infty}\bbP(\calT'(m^2z)-\calT(m^2z)\le \delta m^2z)=0.
$$
But, since the events $(A_z)_{z\in\N}$ are independent with the same probability $p>0$, we have that 
$$
\bbP(\calT'(m^2z)-\calT(m^2z)\le \delta m^2z)\le \bbP({\rm Bin}(p,z)<\delta m^2z),
$$
which tends to zero as $z$ tends to infinity as long as $\delta$ is sufficiently small such that $\delta m^2<p$. This proves the result.
\end{proof}

\begin{proof}[Proof of Lemma~\ref{lem:inverse}]
For Item $(1)$, note that for $s,t\ge 0$, $0\le F_a(t+s)-F_a(t)\le \bbP(\X((a+t,a+t+s])>0)$ which converges to zero as $s$ tends to zero. The right-continuity and monotonicity of the generalized inverse follows from standard arguments. 

The first statement in Item~$(2)$ is a direct consequence of the hitting-probability domination and this also implies the last statement as well as other comparisons. For the monotonicity in the arguments, let us consider $t$ first and use induction. The induction start is clear for $t_1\le t_2$. Now assume that $\tau_{t_1}(u_1,\dots, u_k)\le \tau_{t_2}(u_1,\dots, u_k)$, then, 
\begin{align*}
\tau_{t_1}(u_1,\dots, u_{k+1})= G_{\tau_{t_1}(u_1,\dots, u_k)}(u_{k+1})+\tau_{t_1}(u_1,\dots, u_k),
\end{align*}
and it suffices to show that $a+G_a(u)\le b+G_b(u)$ for $a\le b$. For this, using a change of variable and the fact that $\bbP(\X([a,b+s])>0)\ge \bbP(\X([b,b+s])>0)$,
\begin{align*}
a+G_a(u)&=a+\inf\{s\ge 0\colon u\le \bbP(\X([a,a+s])>0)\}\\
&\le a+\inf\{s\ge b-a\colon u\le \bbP(\X([a,a+s])>0)\}\\
&= a+\inf\{s\ge0\colon u\le \bbP(\X([a,b+s])>0)\}\\
&\le a+\inf\{s\ge0\colon u\le \bbP(\X([b,b+s])>0)\}\\
&\le b+G_b(u).
\end{align*} 
Using this, we can start an iterative proof by verifying that for $u_k\le v_k$ we have that
\begin{align*}
\tau_t(u_1,\dots, u_{k})&=\tau_t(u_1,\dots, u_{k-1})+G_{\tau_t(u_1,\dots, u_{k-1})}(u_k)\\
&\le \tau_t(u_1,\dots, u_{k-1})+G_{\tau_t(u_1,\dots, u_{k-1})}(v_k)\\
&=\tau_t(u_1,\dots,u_{k-1}, v_k).
\end{align*}
Next, for $u_{k-1}\le v_{k-1}$, we can use the monotonicity of $a\mapsto G_a(u)$, to see that
\begin{align*}
\tau_t(u_1,\dots, u_{k})&=\tau_t(u_1,\dots, u_{k-1})+G_{\tau_t(u_1,\dots, u_{k-1})}(u_k)\\
&\le \tau_t(u_1,\dots, v_{k-1})+G_{\tau_t(u_1,\dots, v_{k-1})}(u_k)\\
&=\tau_t(u_1,\dots,u_{k-2},v_{k-1}, u_k).
\end{align*}
Iterating this procedure gives the result. 

For Item $(3)$, fix $i \in \{1,\dots, k\}$. Then by the recursive definition of $\tau_t$ we have 
\begin{align*}
    \tau_t(u_1, \dots, u_i) = G_{\tau_t(u_1, \dots, u_{i-1})}(u_i) + \tau_t(u_1, \dots, u_{i-1}) \geq \tau_t(u_1, \dots, u_{i-1}). 
\end{align*}
Then, we observe
\begin{align*}
    \tau_t(u_1,\dots,u_k) &= \tau_{\tau_t(u_1,\dots,u_{i})}(u_{i+1},\dots,u_k)\\ &\geq \tau_{\tau_t(u_1,\dots,u_{i-1})}(u_{i+1},\dots,u_k) = \tau_t(u_1,\dots,u_{i-1},u_{i+1},\dots,u_k),
\end{align*}
where the inequality follows from monotonicity of $\tau_t$.

For Item~$(4)$, as in the proof of Item~$(2)$ of Proposition~\ref{lem:Shape} above, there exists $\delta>0$ such that $\sup\{\bbP(\X([x,x+\delta))>0)\colon x\in [0,1]\}< u'+\rho$, where $\rho:=(u-u')/2$. Hence, 
$$\tau_t(\underbrace{u,\dots,u}_k)=G_{\tau_t(\underbrace{u,\dots,u}_{k-1})}(u)+\tau_t(\underbrace{u,\dots,u}_{k-1})\ge \delta+\tau_t(\underbrace{u,\dots,u}_{k-1}),$$
since we have that 
$$
G_{\tau_t(u,\dots,u)}(u)=\inf\{t\ge 0\colon \bbP(\X([\tau_t(u,\dots,u),\tau_t(u,\dots,u)+t])>0)\ge u\}\ge \delta.
$$
Hence, $\tau_t(u,\dots,u)\ge \delta k$ and the result follows.

Regarding Item~$(5)$, we can first assume that without loss of generality $t=0$. Note that by definition 
\begin{equation*}
    \begin{split}
        \tau_0(u_1, u_2) 
        &= \tau_0(u_1) + \inf\{s\geq 0\colon u_2 \leq \mathbb{P}(\mathbb{X}([\tau_0(u_1), \tau_0(u_1) + s]) > 0) \} \\
        &\leq 
        \tau_0(u_1) + \inf\{s\geq 0\colon u_2 \leq \mathbb{P}(\mathbb{X}([\tau_0(u_1), \tau_0(u_1) + s]) > 0, \mathbb{X}([0,\tau_0(u_1)])=0) \} \\
        &\leq 
        \tau_0(u_1) + \inf\{s \geq 0\colon u_1 + u_2 \leq \mathbb{P}(\mathbb{X}([0,\tau_0(u_1)+s]) >0)\}
        \\
        &=
        \tau_0(u_1 + u_2), 
    \end{split}
\end{equation*}
where we used that $\tau_0(u_1) = \inf\{s\geq 0\colon u_1 \leq \mathbb{P}(\mathbb{X}([0,s])>0)\}$ for the last equality.

Finally, for Item~$(6)$, by Item~$(1)$, it suffices to check left-continuity of $t\mapsto F_a(t)$, but for $s\ge 0$, since $0\ge F_a(t-s)-F_a(t)\ge -\bbP(\X((a+t-s,a+t])>0)$ and the right-hand side converges to $\bbP(a+t\in \X)=0$ as $s$ tends to zero. On the other hand, for $b\ge 0$, we have that
\begin{align*}
|F_{a-b}(t)-F_a(t)|&=|\bbP(\X([a-b,a-b+t])>0)-\bbP(\X([a,a+t])>0)| \\
&\le 2\bbP(\X([a-b,a))>0)+\bbP(\X((a+t-b,a+t])>0),
\end{align*}
where the first summand on the right-hand side converges to zero as $b$ tends to zero by continuity of measures and the second summand converges to $\bbP(a+t\in \X)$, which equals zero by assumption. The other case $|F_{a+b}(t)-F_a(t)|\to 0$ as $b$ tends to zero follows by the same arguments. 
\end{proof}

\noindent\textbf{Acknowledgement.} The authors acknowledge the financial support of the Leibniz Association within the Leibniz Junior Research Group on \emph{Probabilistic Methods for Dynamic Communication Networks} as part of the Leibniz Competition as well as by Deutsche Forschungsgemeinschaft (DFG, German Research Foundation) under Germany's Excellence Strategy -- The Berlin Mathematics Research Center MATH+ (EXC-2046/1, EXC-2046/2, project ID: 390685689) through the project {\em EF45-3} on {\em Data Transmission in Dynamical Random Networks}.

\bibliographystyle{alpha}
\bibliography{references.bib}

\end{document}